\documentclass[oneside,final,11pt]{amsart}

\usepackage{dsfont}

\newcommand{\R}{\mathds R}
\newcommand{\chilow}[1]{\chi_{\lower2pt\hbox{$\scriptstyle#1$}}}
\newcommand{\Extc}[1]{\Ext\!\big(#1,c_0\!\mskip1.5mu\big)}

\DeclareMathOperator{\Dom}{dom}
\DeclareMathOperator{\Int}{int}
\DeclareMathOperator{\Ker}{Ker}
\DeclareMathOperator{\Ext}{Ext}
\DeclareMathOperator{\supp}{supp}
\DeclareMathOperator{\dens}{dens}

\title[Nontrivial twisted sums of $c_0$ and $C(K)$]{Nontrivial twisted sums of $\mathbf{c_0}$ and $\mathbf{C(K)}$}

\author{Claudia Correa}
\thanks{The first author is sponsored by FAPESP (Process no.\ 2014/00848-2).}
\address{Departamento de Matem\'atica,\hfill\break\indent Universidade de S\~ao Paulo, Brazil}
\email{claudiac.mat@gmail.com}
\author{Daniel V. Tausk}
\address{Departamento de Matem\'atica,\hfill\break\indent Universidade de S\~ao Paulo, Brazil}
\email{tausk@ime.usp.br} \urladdr{http://www.ime.usp.br/\~{}tausk}
\subjclass[2010]{46B20,46E15}
\keywords{Banach spaces of continuous functions; twisted sums of Banach spaces; Valdivia compacta}

\date{October 22nd, 2015}

\begin{document}

\theoremstyle{plain}\newtheorem{teo}{Theorem}[section]
\theoremstyle{plain}\newtheorem{prop}[teo]{Proposition}
\theoremstyle{plain}\newtheorem{lem}[teo]{Lemma}
\theoremstyle{plain}\newtheorem{cor}[teo]{Corollary}
\theoremstyle{definition}\newtheorem{defin}[teo]{Definition}
\theoremstyle{remark}\newtheorem{rem}[teo]{Remark}
\theoremstyle{plain} \newtheorem{assum}[teo]{Assumption}
\theoremstyle{definition}\newtheorem{example}[teo]{Example}
\theoremstyle{plain}\newtheorem*{conjecture}{Conjecture}

\begin{abstract}
We obtain a new class of compact Hausdorff spaces $K$ for which $c_0$ can be nontrivially twisted with $C(K)$.
\end{abstract}

\maketitle

\begin{section}{Introduction}

In this article, we present a broad new class of compact Hausdorff spaces $K$ such that there exists a nontrivial twisted sum of $c_0$ and $C(K)$,
where $C(K)$ denotes the Banach space of continuous real-valued functions on $K$ endowed with the supremum norm. By a {\em twisted sum\/} of the Banach spaces $Y$ and $X$ we mean a short exact sequence $0\rightarrow Y\rightarrow Z\rightarrow X\rightarrow0$, where $Z$ is a Banach space and the maps are bounded linear operators. This twisted sum is called {\em trivial\/} if the exact sequence splits, i.e., if the map $Y\rightarrow Z$ admits a bounded linear left inverse (equivalently, if the map $Z\rightarrow X$ admits a bounded linear right inverse). In other words, the twisted sum is trivial if the range of the map $Y\rightarrow Z$ is complemented in $Z$; in this case, $Z\cong X\oplus Y$. As in \cite{CastilloGonzalez}, we denote by $\Ext(X,Y)$ the set of equivalence classes of twisted sums of $Y$ and $X$ and we write $\Ext(X,Y)=0$ if every such twisted sum is trivial.

Many problems in Banach space theory are related to the quest for conditions under which $\Ext(X,Y)=0$. For instance, an equivalent statement for the classical Theorem of Sobczyk (\cite{JesusSobczyk,Sobczyk}) is that if $X$ is a separable Banach space, then $\Ext(X,c_0)=0$ (\cite[Proposition~3.2]{Advances}). The converse of the latter statement clearly does not hold in general: for example, $\Extc{\ell_1(I)}=0$, since $\ell_1(I)$ is a projective Banach space. However, the following question remains open: {\em is it true that $\Extc{C(K)}\ne0$ for any nonseparable $C(K)$ space?\/} This problem was stated in \cite{CastilloKalton,JesusSobczyk} and further studied in the recent article \cite{Castilloscattered}, in which the author proves that, under the {\em continuum hypothesis\/} (CH), the space $\Extc{C(K)}$ is nonzero for a nonmetrizable compact Hausdorff space $K$ of finite height. In addition to this result, everything else that is known about the problem is summarized in \cite[Proposition~2]{Castilloscattered}, namely that $\Extc{C(K)}$ is nonzero for a $C(K)$ space under any one of the following assumptions:
\begin{itemize}
\item $K$ is a nonmetrizable Eberlein compact space;
\item $K$ is a Valdivia compact space which does not satisfy the countable chain condition (ccc);
\item the weight of $K$ is equal to $\omega_1$ and the dual space of $C(K)$ is not weak*-separable;
\item $K$ has the extension property (\cite{CTextensor}) and it does not have ccc;
\item $C(K)$ contains an isomorphic copy of $\ell_\infty$.
\end{itemize}
Note also that if $\Ext(Y,c_0)\ne0$ and $X$ contains a {\em complemented\/} isomorphic copy of $Y$, then $\Ext(X,c_0)\ne0$.

\smallskip

Here is an overview of the main results of this article. Theorem~\ref{thm:ExtX} gives a condition involving biorthogonal systems in a Banach space $X$ which implies that $\Ext(X,c_0)\ne0$. In the rest of Section~\ref{sec:main}, we discuss some of its implications when $X$ is of the form $C(K)$. It is proven that if $K$ contains a homeomorphic copy of $[0,\omega]\times[0,\mathfrak c]$ or of $2^\mathfrak c$, then $\Extc{C(K)}$ is nonzero, where $\mathfrak c$ denotes the cardinality of the {\it continuum}. In Sections~\ref{sec:Valdivia} and \ref{sec:Chuka}, we investigate the consequences of the results of Section~\ref{sec:main} for Valdivia and Corson compacta. Recall that Valdivia compact spaces constitute a large superclass of Corson compact spaces closed under arbitrary products; moreover, every Eberlein compact is a Corson compact (see \cite{Kalenda} for a survey on Valdivia compacta). Section~\ref{sec:Valdivia} is devoted to the proof that, under CH, it holds that $\Extc{C(K)}\ne0$
for every nonmetrizable Corson compact space $K$. The question of whether $\Extc{C(K)}\ne0$ for an arbitrary nonmetrizable Valdivia compact space $K$ remains open (even under CH), but in Section~\ref{sec:Chuka} we solve some particular cases of this problem.

\end{section}

\begin{section}{General results}
\label{sec:main}

Throughout the paper, the weight and the density character of a topological space $\mathcal X$ are denoted, respectively, by $w(\mathcal X)$ and $\dens(\mathcal X)$. Moreover, we always denote by $\chilow A$ the characteristic function of a set $A$ and by $\vert A\vert$ the cardinality of $A$.
We start with a technical lemma which is the heart of the proof of Theorem~\ref{thm:ExtX}. A family of sets $(A_i)_{i\in I}$ is said to be {\em almost disjoint\/} if each $A_i$ is infinite and $A_i\cap A_j$ is finite, for all $i,j\in I$ with $i\ne j$.
\begin{lem}\label{thm:antichain}
There exists an almost disjoint family $(A_{n,\alpha})_{n\in\omega,\alpha\in\mathfrak c}$ of subsets of $\omega$ satisfying the following property: for every family $(A'_{n,\alpha})_{n\in\omega,\alpha\in\mathfrak c}$ with each $A'_{n,\alpha}\subset A_{n,\alpha}$ cofinite in $A_{n,\alpha}$, it holds that
$\sup_{p\in\omega}\vert M_p\vert=+\infty$, where:
\[M_p=\big\{n\in\omega:\textstyle{p\in\bigcup_{\alpha\in\mathfrak c}A'_{n,\alpha}}\big\}.\]
\end{lem}
\begin{proof}
We will obtain an almost disjoint family $(A_{n,\alpha})_{n\in\omega,\alpha\in\mathfrak c}$ of subsets of $2^{<\omega}$ with the desired property, where
$2^{<\omega}=\bigcup_{k\in\omega}2^k$ denotes the set of finite sequences in $2=\{0,1\}$. For each $\epsilon\in2^\omega$, we set:
\[\mathcal A_\epsilon=\big\{\epsilon\vert_k:k\in\omega\big\},\]
so that $(\mathcal A_\epsilon)_{\epsilon\in2^\omega}$ is an almost disjoint family of subsets of $2^{<\omega}$. Let $(\mathcal B_\alpha)_{\alpha\in\mathfrak c}$ be an enumeration of the uncountable Borel subsets of $2^\omega$. Recalling that $\vert\mathcal B_\alpha\vert=\mathfrak c$ for all $\alpha\in\mathfrak c$
(\cite[Theorem~13.6]{Kechris}), one easily obtains by transfinite recursion a family $(\epsilon_{n,\alpha})_{n\in\omega,\alpha\in\mathfrak c}$ of pairwise distinct elements of $2^\omega$ such that $\epsilon_{n,\alpha}\in\mathcal B_\alpha$, for all $n\in\omega$, $\alpha\in\mathfrak c$. Set $A_{n,\alpha}=\mathcal A_{\epsilon_{n,\alpha}}$ and let $(A'_{n,\alpha})_{n\in\omega,\alpha\in\mathfrak c}$ be as in the statement of the lemma.
For $n\in\omega$, denote by $D_n$ the set of those $\epsilon\in2^\omega$ such that $n\in M_p$ for all but finitely many $p\in\mathcal A_\epsilon$.
Note that:
\[D_n=\bigcup_{k_0\in\omega}\bigcap_{k\ge k_0}\bigcup\big\{C_p:\text{$p\in2^k$ with $n\in M_p$}\big\},\]
where $C_p$ denotes the clopen subset of $2^\omega$ consisting of the extensions of $p$. The above equality implies that $D_n$ is an $F_\sigma$ (and, in particular, a Borel) subset of $2^\omega$. We claim that the complement of $D_n$ in $2^\omega$ is countable. Namely, if it were uncountable, there would exist $\alpha\in\mathfrak c$ with $\mathcal B_\alpha=2^\omega\setminus D_n$. But, since $n\in M_p$ for all $p\in A'_{n,\alpha}$, we have that $\epsilon_{n,\alpha}\in D_n$, contradicting the fact that $\epsilon_{n,\alpha}\in\mathcal B_\alpha$ and proving the claim.
To conclude the proof of the lemma, note that for each $n\ge1$ the intersection $\bigcap_{i<n}D_i$ is nonempty; for $\epsilon\in\bigcap_{i<n}D_i$, we have that $\{i:i<n\}\subset M_p$, for all but finitely many $p\in\mathcal A_\epsilon$.
\end{proof}

Let $X$ be a Banach space. Recall that {\em a biorthogonal system\/} in $X$ is a family $(x_i,\gamma_i)_{i\in I}$ with $x_i\in X$, $\gamma_i\in X^*$, $\gamma_i(x_i)=1$ and $\gamma_i(x_j)=0$ for $i\ne j$. The {\em cardinality\/} of the biorthogonal system $(x_i,\gamma_i)_{i\in I}$ is defined as the cardinality of $I$.
\begin{defin}
Let $(x_i,\gamma_i)_{i\in I}$ be a biorthogonal system in a Banach space $X$. We call $(x_i,\gamma_i)_{i\in I}$ {\em bounded\/} if $\sup_{i\in I}\Vert x_i\Vert<+\infty$ and $\sup_{i\in I}\Vert\gamma_i\Vert<+\infty$; {\em weak*-null\/} if $(\gamma_i)_{i\in I}$ is a weak*-null family, i.e., if $\big(\gamma_i(x)\big)_{i\in I}$ is in $c_0(I)$, for all $x\in X$.
\end{defin}

\begin{teo}\label{thm:ExtX}
Let $X$ be a Banach space. Assume that there exist a weak*-null biorthogonal system $(x_{n,\alpha},\gamma_{n,\alpha})_{n\in\omega,\alpha\in\mathfrak c}$
in $X$ and a constant $C\ge0$ such that:
\[\Big\Vert\sum_{i=1}^kx_{n_i,\alpha_i}\Big\Vert\le C,\]
for all $n_1,\ldots,n_k\in\omega$ pairwise distinct, all $\alpha_1,\ldots,\alpha_k\in\mathfrak c$, and all $k\ge1$.
Then $\Ext(X,c_0)\ne0$.
\end{teo}
\begin{proof}
By \cite[Proposition~1.4.f]{CastilloGonzalez}, we have that $\Ext(X,c_0)=0$ if and only if every bounded operator $T:X\to\ell_\infty/c_0$ admits a {\em lifting\/}\footnote{%
More concretely, a nontrivial twisted sum of $c_0$ and $X$ is obtained by considering the pull-back of the short exact sequence $0\rightarrow c_0\rightarrow\ell_\infty\rightarrow\ell_\infty/c_0\rightarrow0$ by an operator $T:X\to\ell_\infty/c_0$ that does not admit a lifting.}, i.e., a bounded operator $\widehat T:X\to\ell_\infty$ with $T(x)=\widehat T(x)+c_0$, for all $x\in X$. Let us then show that there exists an operator $T:X\to\ell_\infty/c_0$ that does not admit a lifting. To this aim, let $(A_{n,\alpha})_{n\in\omega,\alpha\in\mathfrak c}$ be an almost disjoint family as in Lemma~\ref{thm:antichain} and consider the unique isometric embedding $S:c_0(\omega\times\mathfrak c)\to\ell_\infty/c_0$ such that $S(e_{n,\alpha})=\chilow{A_{n,\alpha}}+c_0$, where $(e_{n,\alpha})_{n\in\omega,\alpha\in\mathfrak c}$ denotes the canonical basis of $c_0(\omega\times\mathfrak c)$. Denote by $\Gamma:X\to c_0(\omega\times\mathfrak c)$ the bounded operator with coordinate functionals $(\gamma_{n,\alpha})_{n\in\omega,\alpha\in\mathfrak c}$ and set $T=S\circ\Gamma:X\to\ell_\infty/c_0$. Assuming by contradiction that there exists a lifting $\widehat T$ of $T$ and denoting by $(\mu_p)_{p\in\omega}$ the sequence of coordinate functionals of $\widehat T$, we have that the set:
\[A'_{n,\alpha}=\big\{p\in A_{n,\alpha}:\mu_p(x_{n,\alpha})\ge\tfrac12\big\}\]
is cofinite in $A_{n,\alpha}$. It follows that for each $k\ge1$, there exist $p\in\omega$, $n_1,\ldots,n_k\in\omega$ pairwise distinct,
and $\alpha_1,\ldots,\alpha_k\in\mathfrak c$ such that $p\in A'_{n_i,\alpha_i}$, for $i=1,\ldots,k$. Hence:
\[\frac k2\le\mu_p\Big(\sum_{i=1}^kx_{n_i,\alpha_i}\Big)\le C\Vert\widehat T\Vert,\]
which yields a contradiction.
\end{proof}

\begin{cor}\label{thm:produtozero}
Let $K$ be a compact Hausdorff space. Assume that there exists a bounded weak*-null biorthogonal system $(f_{n,\alpha},\gamma_{n,\alpha})_{n\in\omega,\alpha\in\mathfrak c}$ in $C(K)$ such that $f_{n,\alpha}f_{m,\beta}=0$, for all $n,m\in\omega$ with $n\ne m$ and all $\alpha,\beta\in\mathfrak c$. Then $\Extc{C(K)}\ne0$.\qed
\end{cor}

\begin{defin}\label{thm:propstar}
We say that a compact Hausdorff space $K$ satisfies {\em property ($*$)\/} if there exist a sequence $(F_n)_{n\in\omega}$ of closed subsets of $K$ and a bounded weak*-null biorthogonal system $(f_{n,\alpha},\gamma_{n,\alpha})_{n\in\omega,\alpha\in\mathfrak c}$ in $C(K)$ such that:
\begin{equation}\label{eq:discretefamily}
F_n\cap\overline{\bigcup_{m\ne n}F_m}=\emptyset
\end{equation}
and $\supp f_{n,\alpha}\subset F_n$, for all $n\in\omega$ and all $\alpha\in\mathfrak c$, where $\supp f_{n,\alpha}$ denotes the support of $f_{n,\alpha}$.
\end{defin}

In what follows, we denote by $M(K)$ the space of finite countably-additive signed regular Borel measures on $K$, endowed with the total variation norm. We identify as usual the dual space of $C(K)$ with $M(K)$.
\begin{lem}\label{thm:fechapracima}
Let $K$ be a compact Hausdorff space and $L$ be a closed subspace of $K$. If $L$ satisfies property ($*$), then so does $K$.
\end{lem}
\begin{proof}
Consider, as in Definition~\ref{thm:propstar}, a sequence $(F_n)_{n\in\omega}$ of closed subsets of $L$ and a bounded weak*-null biorthogonal system $(f_{n,\alpha},\gamma_{n,\alpha})_{n\in\omega,\alpha\in\mathfrak c}$ in $C(L)$.
By recursion on $n$, one easily obtains a sequence $(U_n)_{n\in\omega}$ of pairwise disjoint open subsets of $K$ with each $U_n$ containing $F_n$.
Let $V_n$ be an open subset of $K$ with $F_n\subset V_n\subset\overline{V_n}\subset U_n$. Using Tietze's Extension Theorem and Urysohn's Lemma, we get a continuous extension $\tilde f_{n,\alpha}$ of $f_{n,\alpha}$ to $K$ with support contained in $\overline{V_n}$ and having the same norm as $f_{n,\alpha}$. To conclude the proof, let
$\tilde\gamma_{n,\alpha}\in M(K)$ be the extension of $\gamma_{n,\alpha}\in M(L)$ that vanishes identically outside of $L$
and observe that $(\tilde f_{n,\alpha},\tilde\gamma_{n,\alpha})_{n\in\omega,\alpha\in\mathfrak c}$ is a bounded weak*-null biorthogonal system in $C(K)$.
\end{proof}

As an immediate consequence of Lemma~\ref{thm:fechapracima} and Corollary~\ref{thm:produtozero}, we obtain the following result.
\begin{teo}
If a compact Hausdorff space $L$ satisfies property ($*$), then every compact Hausdorff space $K$ containing a homeomorphic copy of $L$ satisfies
$\Extc{C(K)}\ne0$.\qed
\end{teo}

We now establish a few results which give sufficient conditions for a space $K$ to satisfy property ($*$).
Recall that, given a closed subset $F$ of a compact Hausdorff space $K$,
an {\em extension operator\/} for $F$ in $K$ is a bounded operator $E:C(F)\to C(K)$ which is a right inverse for the restriction operator $C(K)\ni f\mapsto f\vert_F\in C(F)$. Note that $F$ admits an extension operator in $K$ if and only if the kernel
\[C(K|F)=\big\{f\in C(K):f\vert_F=0\big\}\]
of the restriction operator is complemented in $C(K)$. A point $x$ of a topological space $\mathcal X$ is called a {\em cluster point\/} of a sequence $(S_n)_{n\in\omega}$ of subsets of $\mathcal X$ if every neighborhood of $x$ intersects $S_n$ for infinitely many $n\in\omega$.

\begin{lem}\label{thm:abc}
Let $K$ be a compact Hausdorff space. Assume that there exist a sequence $(F_n)_{n\in\omega}$ of pairwise disjoint closed subsets of $K$ and a
closed subset $F$ of $K$ satisfying the following conditions:
\begin{itemize}
\item[(a)] $F$ admits an extension operator in $K$;
\item[(b)] every cluster point of $(F_n)_{n\in\omega}$ is in $F$ and $F_n\cap F=\emptyset$, for all $n\in\omega$;
\item[(c)] there exists a family $(f_{n,\alpha},\gamma_{n,\alpha})_{n\in\omega,\alpha\in\mathfrak c}$, where
$(f_{n,\alpha},\gamma_{n,\alpha})_{\alpha\in\mathfrak c}$ is a weak*-null biorthogonal system in $C(F_n)$ for each $n\in\omega$ and \[\sup_{n\in\omega,\alpha\in\mathfrak c}\Vert f_{n,\alpha}\Vert<+\infty,\quad\sup_{n\in\omega,\alpha\in\mathfrak c}\Vert\gamma_{n,\alpha}\Vert<+\infty.\]
\end{itemize}
Then $K$ satisfies property ($*$).
\end{lem}
\begin{proof}
From (b) and the fact that the $F_n$ are pairwise disjoint it follows that \eqref{eq:discretefamily} holds. Let $(U_n)_{n\in\omega}$, $(V_n)_{n\in\omega}$, and $(\tilde f_{n,\alpha})_{n\in\omega,\alpha\in\mathfrak c}$ be as in the proof of Lemma~\ref{thm:fechapracima}; we assume also that $\overline V_n\cap F=\emptyset$, for all $n\in\omega$. Let $\tilde\gamma_{n,\alpha}\in M(K)$ be the
extension of $\gamma_{n,\alpha}\in M(F_n)$ that vanishes identically outside of $F_n$. We have that $(\tilde f_{n,\alpha},\tilde\gamma_{n,\alpha})_{n\in\omega,\alpha\in\mathfrak c}$ is a bounded biorthogonal system in $C(K)$ and that $(\tilde\gamma_{n,\alpha})_{\alpha\in\mathfrak c}$ is weak*-null for each $n$, though it is not true in general that the entire family $(\tilde\gamma_{n,\alpha})_{n\in\omega,\alpha\in\mathfrak c}$ is weak*-null. In order to take care of this problem, let $P:C(K)\to C(K|F)$ be a bounded projection and set $\hat\gamma_{n,\alpha}=\tilde\gamma_{n,\alpha}\circ P$. Since all $\tilde f_{n,\alpha}$ are in $C(K|F)$, we have that
$(\tilde f_{n,\alpha},\hat\gamma_{n,\alpha})_{n\in\omega,\alpha\in\mathfrak c}$ is biorthogonal. To prove that $(\hat\gamma_{n,\alpha})_{n\in\omega,\alpha\in\mathfrak c}$ is weak*-null, note that (b) implies that $\lim_{n\to+\infty}\Vert f\vert_{F_n}\Vert=0$, for
all $f\in C(K|F)$.
\end{proof}

\begin{cor}\label{thm:sequencebiorthogonal}
Let $K$ be a compact Hausdorff space. If $C(K)$ admits a bounded weak*-null biorthogonal system of cardinality $\mathfrak c$, then
the space $[0,\omega]\times K$ satisfies property ($*$). In particular, $L\times K$ satisfies property ($*$) for every compact Hausdorff space
$L$ containing a nontrivial convergent sequence.
\end{cor}
\begin{proof}
Take $F_n=\{n\}\times K$, $F=\{\omega\}\times K$, and use the fact that $F$ is a retract of $[0,\omega]\times K$ and thus admits an extension operator in $[0,\omega]\times K$.
\end{proof}

\begin{cor}\label{thm:0omega0c}
The spaces $[0,\omega]\times[0,\mathfrak c]$ and $2^\mathfrak c$ satisfy property ($*$). In particular, a product of at least $\mathfrak c$ compact Hausdorff spaces with more than one point satisfies property ($*$).
\end{cor}
\begin{proof}
The family $(\chilow{[0,\alpha]},\delta_\alpha-\delta_{\alpha+1})_{\alpha\in\mathfrak c}$ is a bounded weak*-null biorthogonal system in $C\big([0,\mathfrak c]\big)$, where $\delta_\alpha\in M\big([0,\mathfrak c]\big)$ denotes the probability measure with support $\{\alpha\}$.
It follows from Corollary~\ref{thm:sequencebiorthogonal} that $[0,\omega]\times[0,\mathfrak c]$ satisfies property ($*$). To see that $2^\mathfrak c$ also does, note that the map $[0,\mathfrak c]\ni\alpha\mapsto\chilow\alpha\in2^\mathfrak c$ embeds $[0,\mathfrak c]$ into $2^\mathfrak c$, so that $2^{\mathfrak c}\cong2^\omega\times2^{\mathfrak c}$ contains a homeomorphic copy of $[0,\omega]\times[0,\mathfrak c]$.
\end{proof}

Recall that a {\em projectional resolution of the identity\/} (PRI) of a Banach space $X$ is a family $(P_\alpha)_{\omega\le\alpha\le\dens(X)}$ of
projection operators $P_\alpha:X\to X$ satisfying the following conditions:
\begin{itemize}
\item $\Vert P_\alpha\Vert=1$, for $\omega\le\alpha\le\dens(X)$;
\item $P_{\dens(X)}$ is the identity of $X$;
\item $P_\alpha[X]\subset P_\beta[X]$ and $\Ker(P_\beta)\subset\Ker(P_\alpha)$, for $\omega\le\alpha\le\beta\le\dens(X)$;
\item $P_\alpha(x)=\lim_{\beta<\alpha}P_\beta(x)$, for all $x\in X$, if $\omega<\alpha\le\dens(X)$ is a limit ordinal;
\item $\dens\big(P_\alpha[X]\big)\le\vert\alpha\vert$, for $\omega\le\alpha\le\dens(X)$.
\end{itemize}
We call the PRI {\em strictly increasing\/} if $P_\alpha[X]$ is a proper subspace of $P_\beta[X]$, for $\omega\le\alpha<\beta\le\dens(X)$.

\begin{cor}\label{thm:PRIstrict}
Let $K$ and $L$ be compact Hausdorff spaces such that $L$ contains a nontrivial convergent sequence and $w(K)\ge\mathfrak c$. If $C(K)$ admits a strictly increasing PRI, then the space $L\times K$ satisfies property ($*$).
\end{cor}
\begin{proof}
This follows from Corollary~\ref{thm:sequencebiorthogonal} by observing that if a Banach space $X$ admits a strictly increasing PRI, then $X$ admits a weak*-null biorthogonal system $(x_\alpha,\gamma_\alpha)_{\omega\le\alpha<\dens(X)}$ with $\Vert x_\alpha\Vert=1$ and $\Vert\gamma_\alpha\Vert\le2$, for all $\alpha$. Namely, pick a unit vector $x_\alpha$ in  $P_{\alpha+1}[X]\cap\Ker(P_\alpha)$ and set $\gamma_\alpha=\phi_\alpha\circ(P_{\alpha+1}-P_\alpha)$, where $\phi_\alpha\in X^*$ is a norm-one functional satisfying $\phi_\alpha(x_\alpha)=1$.
\end{proof}

\end{section}

\begin{section}{Nontrivial twisted sums for Corson compacta}
\label{sec:Valdivia}

Let us recall some standard definitions. Given an index set $I$, we write $\Sigma(I)=\big\{x\in\R^I:\text{$\supp x$ is countable}\big\}$, where the support $\supp x$ of $x$ is defined by $\supp x=\big\{i\in I:x_i\ne0\big\}$. Given a compact Hausdorff space $K$, we call $A$
a {\em $\Sigma$-subset\/} of $K$ if there exist an index set $I$ and a continuous injection $\varphi:K\to\R^I$ such that $A=\varphi^{-1}[\Sigma(I)]$.
The space $K$ is called a {\em Valdivia compactum\/} if it admits a dense $\Sigma$-subset and it is called a {\em Corson compactum\/} if $K$ is a $\Sigma$-subset of itself. This section will be dedicated to the proof of the following result.

\begin{teo}\label{thm:corCorson}
If $K$ is a Corson compact space with $w(K)\ge\mathfrak c$, then $\Extc{C(K)}\ne0$. In particular, under CH, we have $\Extc{C(K)}\ne0$ for any nonmetrizable Corson compact space $K$.
\end{teo}

The fact that $\Extc{C(K)}\ne0$ for a Valdivia compact space $K$ which does not have ccc is already known (\cite[Proposition~2]{Castilloscattered}).
Our strategy for the proof of Theorem~\ref{thm:corCorson} is to use Lemma~\ref{thm:abc} to show that if $K$ is a Corson compact space with $w(K)\ge\mathfrak c$ having ccc, then $K$ satisfies property ($*$). We start with a lemma that will be used as a tool for verifying the assumptions of Lemma~\ref{thm:abc}.
Recall that a closed subset of a topological space is called {\em regular\/} if it is the closure of an open set (equivalently, if it is the closure
of its own interior). Obviously, a closed subset of a Corson compact space is again Corson and a regular closed subset of a Valdivia compact space is again
Valdivia.
\begin{lem}\label{thm:lemanonopen}
Let $K$ be a compact Hausdorff space and $F$ be a closed non-open $G_\delta$ subset of $K$. Then there exists a sequence $(F_n)_{n\in\omega}$
of nonempty pairwise disjoint regular closed subsets of $K$ such that condition~(b) in the statement of Lemma~\ref{thm:abc} holds.
\end{lem}
\begin{proof}
We can write $F=\bigcap_{n\in\omega}V_n$, with each $V_n$ open in $K$ and $\overline{V_{n+1}}$ properly contained in $V_n$.
Set $U_n=V_n\setminus\overline{V_{n+1}}$, so that all cluster points of $(U_n)_{n\in\omega}$ are in $F$. To conclude the proof, let $F_n$ be a nonempty regular closed set contained in $U_n$.
\end{proof}

Once we get the closed sets $(F_n)_{n\in\omega}$ from Lemma~\ref{thm:lemanonopen}, we still have to verify the rest of the conditions in the statement of Lemma~\ref{thm:abc}. First, we need an assumption ensuring that $w(F_n)\ge\mathfrak c$, for all $n$. To this aim, given a point $x$ of a topological space $\mathcal X$, we define the {\em weight of $x$ in $\mathcal X$\/} by:
\[w(x,\mathcal X)=\min\big\{w(V):\text{$V$ neighborhood of $x$ in $\mathcal X$}\big\}.\]
Recall that if $K$ is a Valdivia compact space, then $C(K)$ admits a PRI (\cite[Theorem~2]{Valdivia}). Moreover, a trivial adaptation of the proof in \cite{Valdivia} shows in fact that $C(K)$ admits a strictly increasing PRI. Thus, by the argument in the proof of
Corollary~\ref{thm:PRIstrict}, $C(K)$ admits a weak*-null biorthogonal system $(f_\alpha,\gamma_\alpha)_{\omega\le\alpha<w(K)}$ such that
$\Vert f_\alpha\Vert\le1$ and $\Vert\gamma_\alpha\Vert\le2$, for all $\alpha$. The following result is now immediately obtained.
\begin{cor}\label{thm:Gdeltastar}
Let $K$ be a Valdivia compact space such that $w(x,K)\ge\mathfrak c$, for all $x\in K$. Assume that there exists a closed non-open $G_\delta$ subset $F$ admitting an extension operator in $K$. Then $K$ satisfies property ($*$).\qed
\end{cor}

Assuming that $K$ has ccc, the next lemma allows us to reduce the proof of Theorem~\ref{thm:corCorson} to the case when $w(x,K)\ge\mathfrak c$, for all $x\in K$.
\begin{lem}\label{thm:cccH}
Let $K$ be a ccc Valdivia compact space and set:
\[H=\big\{x\in K:w(x,K)\ge\mathfrak c\big\}.\]
Then:
\begin{itemize}
\item[(a)] $H\ne\emptyset$, if $w(K)\ge\mathfrak c$;
\item[(b)] $w\big(K\setminus\Int(H)\big)<\mathfrak c$, where $\Int(H)$ denotes the interior of $H$;
\item[(c)] $H$ is a regular closed subset of $K$;
\item[(d)] $w(x,H)\ge\mathfrak c$, for all $x\in H$.
\end{itemize}
\end{lem}
\begin{proof}
If $H=\emptyset$, then $K$ can be covered by a finite number of open sets with weight less than $\mathfrak c$, so that $w(K)<\mathfrak c$. This proves (a).
To prove (b), let $(U_i)_{i\in I}$ be maximal among antichains of open subsets of $K$ with weight less than $\mathfrak c$. Since $I$ is countable and
$\mathfrak c$ has uncountable cofinality, we have that $U=\bigcup_{i\in I}U_i$ has weight less than $\mathfrak c$. From the maximality of $(U_i)_{i\in I}$,
it follows that $K\setminus H\subset\overline U$; then $K\setminus\Int(H)=\overline{K\setminus H}\subset\overline U$. To conclude the proof of (b), let us show that $w(\overline U)<\mathfrak c$. Let $A$ be a dense $\Sigma$-subset of $K$ and let $D$ be a dense subset of $A\cap U$ with $\vert D\vert\le w(U)$. Then $\overline D$ is homeomorphic to a subspace of $\R^{w(U)}$, so that $w(\overline U)=w(\overline D)\le w(U)<\mathfrak c$. To prove (c), note that $H$ is clearly closed; moreover, by (b), the open set $K\setminus\overline{\Int(H)}$ has weight less than $\mathfrak c$ and hence it is contained in $K\setminus H$.
Finally, to prove (d), let $V$ be a closed neighborhood in $K$ of some $x\in H$. By (b), we have $w(V\setminus H)<\mathfrak c$. Recall from \cite[p.\ 26]{Hodel} that if a compact Hausdorff space is the union of not more than $\kappa$ subsets of weight not greater than $\kappa$, then the weight of the space is not greater than $\kappa$. Since $w(V)\ge\mathfrak c$, it follows from such result that
$w(V\cap H)\ge\mathfrak c$.
\end{proof}

\begin{proof}[Proof of Theorem~\ref{thm:corCorson}]
By Lemma~\ref{thm:cccH}, it suffices to prove that if $K$ is a nonempty Corson compact space such that $w(x,K)\ge\mathfrak c$ for all $x\in K$, then $K$ satisfies property ($*$). Since a nonempty Corson compact space $K$ admits a $G_\delta$ point $x$ (\cite[Theorem~3.3]{Kalenda}), this fact follows from
Corollary~\ref{thm:Gdeltastar} with $F=\{x\}$.
\end{proof}

\begin{rem}
It is known that under Martin's Axiom (MA) and the negation of CH, every ccc Corson compact is metrizable (\cite{ArgyrosMartin}).
Thus, Theorem~\ref{thm:corCorson} implies that $\Extc{C(K)}\ne0$ for every nonmetrizable Corson compact space $K$ under MA.
\end{rem}

\end{section}

\begin{section}{Towards the general Valdivia case}\label{sec:Chuka}

In this section we prove that $\Extc{C(K)}\ne0$ for certain classes of nonmetrizable Valdivia compact spaces $K$ and we propose a strategy for dealing with the general problem. First, let us state some results which are immediate consequences of what we have done so far.

\begin{prop}\label{thm:productValdivia}
If $K$ is a Valdivia compact space with $w(K)\ge\mathfrak c$ and $L$ is a compact Hausdorff space containing a nontrivial convergent sequence, then $L\times K$ satisfies property ($*$).
\end{prop}
\begin{proof}
As we have observed in Section~\ref{sec:Valdivia}, if $K$ is a Valdivia compact space, then $C(K)$ admits a strictly increasing PRI. The conclusion follows from Corollary~\ref{thm:PRIstrict}.
\end{proof}

\begin{prop}\label{thm:pontopesadoGdelta}
Let $K$ be a Valdivia compact space admitting a $G_\delta$ point $x$ with $w(x,K)\ge\mathfrak c$. Then $\Extc{C(K)}\ne0$ and,
if $K$ has ccc, then $K$ satisfies property ($*$).
\end{prop}
\begin{proof}
As mentioned before, the non-ccc case is already known. Assuming that $K$ has ccc, define $H$ as in Lemma~\ref{thm:cccH} and conclude that $H$ satisfies property ($*$) using Corollary~\ref{thm:Gdeltastar} with $F=\{x\}$.
\end{proof}

\begin{cor}
Let $K$ be a Valdivia compact space with $w(K)\ge\mathfrak c$ admitting a dense $\Sigma$-subset $A$ such that $K\setminus A$ is of first category.
Then $\Extc{C(K)}\ne0$ and, if $K$ has ccc, then $K$ satisfies property ($*$).
\end{cor}
\begin{proof}
By \cite[Theorem~3.3]{Kalenda}, $K$ has a dense subset of $G_\delta$ points. Assuming that $K$ has ccc and defining $H$ as in Lemma~\ref{thm:cccH}, we obtain that $H$ contains a $G_\delta$ point of $K$, which implies that $K$ satisfies the assumptions of Proposition~\ref{thm:pontopesadoGdelta}.
\end{proof}

Now we investigate conditions under which a Valdivia compact space $K$ contains a homeomorphic copy of $[0,\omega]\times[0,\mathfrak c]$. Given an index set $I$ and a subset $J$ of $I$, we denote by $r_J:\R^I\to\R^I$ the map defined by setting $r_J(x)\vert_J=x\vert_J$ and $r_J(x)\vert_{I\setminus J}\equiv0$, for all $x\in\R^I$. Following \cite{ArgyrosCorson}, given a subset $K$ of $\R^I$, we say that $J\subset I$ is {\em $K$-good\/} if $r_J[K]\subset K$. In \cite[Lemma~1.2]{ArgyrosCorson}, it is proven that if $K$ is a compact subset of $\R^I$ and $\Sigma(I)\cap K$ is dense in $K$, then every infinite subset $J$ of $I$ is contained in a $K$-good set $J'$ with $\vert J\vert=\vert J'\vert$.

\begin{prop}\label{thm:sequencia}
Let $K$ be a Valdivia compact space admitting a dense $\Sigma$-subset $A$ such that some point of $K\setminus A$ is the limit of a nontrivial
sequence in $K$. Then $K$ contains a homeomorphic copy of\/ $[0,\omega]\times[0,\omega_1]$. In particular, assuming CH, we have that $K$ satisfies property ($*$).
\end{prop}
\begin{proof}
We can obviously assume that $K$ is a compact subset of some $\R^I$ and that $A=\Sigma(I)\cap K$. Since $A$ is sequentially closed, our hypothesis implies that there exists a continuous injective map $[0,\omega]\ni n\mapsto x_n\in K\setminus A$. Let $J$ be a countable subset of $I$ such that $x_n\vert_J\ne x_m\vert_J$, for all $n,m\in[0,\omega]$ with $n\ne m$. Using \cite[Lemma~1.2]{ArgyrosCorson} and transfinite recursion, one easily obtains a family $(J_\alpha)_{\alpha\le\omega_1}$ of $K$-good subsets of $I$ satisfying the following conditions:
\begin{itemize}
\item $J_\alpha$ is countable, for $\alpha<\omega_1$;
\item $J\subset J_0$;
\item $J_\alpha\subset J_\beta$, for $0\le\alpha\le\beta\le\omega_1$;
\item $J_\alpha=\bigcup_{\beta<\alpha}J_\beta$, for limit $\alpha\in[0,\omega_1]$;
\item for all $n\in[0,\omega]$, the map $[0,\omega_1]\ni\alpha\mapsto J_\alpha\cap\supp x_n$ is injective.
\end{itemize}
Given these conditions, it is readily checked that the map
\[[0,\omega]\times[0,\omega_1]\ni(n,\alpha)\longmapsto r_{J_\alpha}(x_n)\in K\]
is continuous and injective.
\end{proof}

\begin{rem}
The following converse of Proposition~\ref{thm:sequencia} also holds: if $K$ is a Valdivia compact space containing a homeomorphic copy of $[0,\omega]\times[0,\omega_1]$, then $K\setminus A$ contains a nontrivial convergent sequence, for {\em any\/} dense $\Sigma$-subset $A$ of $K$.
Namely, if $\phi:[0,\omega]\times[0,\omega_1]\to K$ is a continuous injection, then $\phi(\omega,\alpha)\in K\setminus A$ for some $\alpha\in[0,\omega_1]$, since $[0,\omega_1]$ is not Corson (\cite[Example~1.10~(i)]{Kalenda}). Moreover, the nontrivial sequence $\big(\phi(n,\alpha)\big)_{n\in\omega}$ converges
to $\phi(\omega,\alpha)$. One consequence of this observation is that if $K\setminus A$ contains a nontrivial convergent sequence for {\em some\/} dense $\Sigma$-subset $A$ of $K$, then $K\setminus A$ contains a nontrivial convergent sequence for {\em any\/} dense $\Sigma$-subset $A$ of $K$.
\end{rem}

\begin{rem}
If a Valdivia compact space $K$ admits two distinct dense $\Sigma$-subsets, then the assumption of Proposition~\ref{thm:sequencia} holds for $K$.
Namely, given dense $\Sigma$-subsets $A$ and $B$ of $K$ and a point $x\in A\setminus B$, then $x$ is not isolated, since $B$ is dense.
Moreover, $x$ is not isolated in $A$, because $A$ is dense. Finally, since $A$ is a Fr\'echet--Urysohn space (\cite[Lemma~1.6~(ii)]{Kalenda}), $x$ is the limit of a sequence in $A\setminus\{x\}$.
\end{rem}

Finally, we observe that the validity of the following conjecture would imply, under CH, that $\Extc{C(K)}\ne0$ for any nonmetrizable Valdivia compact space $K$.
\begin{conjecture}
If $K$ is a nonempty Valdivia compact space having ccc, then either $K$ has a $G_\delta$ point or $K$ admits a nontrivial convergent sequence in the complement of a dense $\Sigma$-subset.
\end{conjecture}
To see that the conjecture implies the desired result, use Lemma~\ref{thm:cccH} and Propositions~\ref{thm:pontopesadoGdelta} and \ref{thm:sequencia}, keeping in mind that a regular closed subset of a ccc space has ccc as well. The conjecture remains open, but in what follows we present an example showing that it is false if the assumption that $K$ has ccc is removed.

\smallskip

Recall that a {\em tree\/} is a partially ordered set $(T,{\le})$ such that, for all $t\in T$, the set $(\cdot,t)=\big\{s\in T:s<t\big\}$ is well-ordered.
As in \cite[p.~288]{Todorcevic}, we define a compact Hausdorff space from a tree $T$ by considering the subspace $P(T)$ of $2^T$ consisting of all characteristic functions of paths of $T$; by a {\em path\/} of $T$ we mean a totally ordered subset $A$ of $T$ such that $(\cdot,t)\subset A$, for all $t\in A$. It is easy to see that $P(T)$ is closed in $2^T$; we call it the {\em path space\/} of $T$.

Denote by $S(\omega_1)$ the set of countable successor ordinals and consider the tree
$T=\bigcup_{\alpha\in S(\omega_1)}\omega_1^\alpha$, partially ordered by inclusion. The path space $P(T)$ is the image of the injective map $\Lambda\ni\lambda\mapsto\chilow{A(\lambda)}\in2^T$, where $\Lambda=\bigcup_{\alpha\le\omega_1}\omega_1^\alpha$ and $A(\lambda)=\big\{t\in T:t\subset\lambda\big\}$.
\begin{prop}
If the tree $T$ is defined as above, then its path space $P(T)$ is a compact subspace of $\R^T$ satisfying the following conditions:
\begin{itemize}
\item[(a)] $P(T)\cap\Sigma(T)$ is dense in $P(T)$, so that $P(T)$ is Valdivia;
\item[(b)] $P(T)$ has no $G_\delta$ points;
\item[(c)] no point of $P(T)\setminus\Sigma(T)$ is the limit of a nontrivial sequence in $P(T)$.
\end{itemize}
\end{prop}
\begin{proof}
To prove (a), note that $\chilow{A(\lambda)}=\lim_{\alpha<\omega_1}\chilow{A(\lambda\vert_\alpha)}$
for all $\lambda\in\omega_1^{\omega_1}$. Let us prove (b). Since $P(T)$ is Valdivia, every $G_\delta$ point of $P(T)$ must be in $\Sigma(T)$ (\cite[Proposition~2.2~(3)]{Kalenda2}), i.e., it must be of the form $\chilow{A(\lambda)}$, with $\lambda\in\omega_1^\alpha$, $\alpha<\omega_1$. To see that $\chilow{A(\lambda)}$ cannot be a $G_\delta$ point of $P(T)$, it suffices to check that for any countable subset $E$ of $T$, there exists $\mu\in\Lambda$, $\mu\ne\lambda$, such that $\chilow{A(\lambda)}$ and $\chilow{A(\mu)}$ are identical on $E$. To this aim, simply take $\mu=\lambda\cup\{(\alpha,\beta)\}$, with $\beta\in\omega_1\setminus\big\{t(\alpha):\text{$t\in E$ and $\alpha\in\Dom(t)$}\big\}$. Finally, to prove (c), let $\big(\chilow{A(\lambda_n)}\big)_{n\ge1}$ be a sequence of pairwise distinct elements of $P(T)$ converging to some $\epsilon\in P(T)$ and note that the support of $\epsilon$ must be contained in the countable set $\bigcup_{n\ne m}\big(A(\lambda_n)\cap A(\lambda_m)\big)$.
\end{proof}

It is easy to see that, for $T$ defined as above, the space $P(T)$ does not have ccc. Namely, setting $U_t=\big\{\epsilon\in P(T):\epsilon(t)=1\big\}$ for $t\in T$, we have that $U_t$ is a nonempty open subset of $P(T)$ and that $U_t\cap U_s=\emptyset$, when $t,s\in T$ are incomparable.

\end{section}

\end{document}